\documentclass[numeric-verb]{article}
\usepackage{hyperref}
\usepackage{algorithm}
\usepackage{algpseudocode}
\usepackage{amssymb}
\usepackage[T1]{fontenc}
\usepackage{amsmath}

\usepackage{mathtools}
\DeclarePairedDelimiter{\ceil}{\lceil}{\rceil}

\usepackage[a4paper, total={6.2in, 8.3in}]{geometry}
\usepackage{threeparttable}

\newtheorem{theorem}{Theorem}
\newtheorem{proof}{Proof}
\newtheorem{lemma}{Lemma}
\usepackage[table]{xcolor}
\newcommand{\headrow}{\rowcolor{black!20}}

\newcommand{\N}{\mathbb{N}}

\newcommand{\R}{\mathbb{R}}
\newcommand{\C}{\mathbb{C}}
\newcommand{\dep}{\mbox{dep}}
\newcommand{\TT}{^\ast} %fÃŒr Adjungierte/Transponierte

\title{Krylov Methods for Adjoint-Free Singular Vector Based Perturbations in Dynamical Systems }

\author{
    Jens Winkler\\
    \small Deutscher Wetterdienst and\\
    \small Numerics and Optimization\\ 
    \small Philipps Universit\"at Marburg\\
    \small winklerj@mathematik.uni-marburg.de
  \and
  
    Michael Denhard\\
    \small Deutscher Wetterdienst\\[3mm]
    
    Bernhard A. Schmitt\\    
    \small Numerics and Optimization\\
    \small Philipps Universit\"at Marburg\\
}
\date{September 2019}

\begin{document}
\maketitle

\begin{abstract}
The estimation of weather forecast uncertainty with ensemble systems requires a careful selection of perturbations to establish a reliable sampling of the error growth potential in the phase space of the model. Usually, the singular vectors of the tangent linear model propagator are used to identify the fastest growing modes (classical singular vector perturbation (SV) method). In this paper we present an efficient matrix-free block Krylov method for generating fast growing perturbations in high dimensional dynamical systems. A specific matrix containing the non-linear evolution of perturbations is introduced, which we call Evolved Increment Matrix (EIM). Instead of solving an equivalent eigenvalue problem, we use the Arnoldi method for a direct approximation of the leading singular vectors of this matrix, which however is never computed explicitly. This avoids linear and adjoint models but requires forecasts with the full non-linear system. The performance of the approximated perturbations is compared with singular vectors of a full EIM (not with the classical SV method). We show promising results for the Lorenz96 differential equations and a shallow water model, where we obtain good approximations of the fastest growing perturbations by using only a small number of Arnoldi iterations.\\
%% Please include a maximum of seven keywords
\textbf{ keywords: numerical weather prediction, ensemble systems, forecast uncertainty, singular vectors, Arnoldi method, conditional non-linear optimal perturbations}%\emph{keyword 2}
\end{abstract}

\section{Introduction}

Weather forecasting is a computationally challenging initial value problem. It involves the estimation of the current atmospheric state and running a complex numerical weather prediction (NWP) model. Forecast errors have two major sources: model error and errors in initial conditions. While the former leads to deviations from the true atmospheric behaviour, the latter may amplify quickly at any time during the forecast due to non-linear dynamical processes in the model. These instabilities may not exactly correspond to the sensitivities in the real atmosphere, but it would be quite revealing to understand which fraction of forecast uncertainty can be attributed to the errors in initial conditions.

For the estimation of forecast uncertainty \cite{bib:LEITH} introduced a Monte Carlo ansatz computing not just one but several randomly perturbed forecasts in order to estimate a probability density function. Given the fact that usually only a small number of forecast integrations with a complex model is possible, the strategy of sampling the initial uncertainty in high dimensional phase spaces by random perturbations is quite inefficient. 
A similar argument holds for perturbations from existing ensemble data assimilation (EDA) schemes. For example the 4DVar based EDA at ECMWF (\cite{bib:Isaksen}) selects valid model trajectories consistent with sets of perturbed observations, but does not specifically look for fast separating trajectories. Ensemble Kalman filters (EnKF) use various types of co-variance inflation techniques, where most are either non-flow dependent or of random nature. \cite{bib:Hamill} discuss the constraints on spread growth in forecasts initialized from EnKF. 
Of course, analysis error co-variances are the framework of any useful perturbation strategy. They provide an estimate of the initial uncertainty given the amount and quality of local observations. But the same analysis error co-variances could be generated from a large number of different sets of perturbations. EDA schemes do not systematically analyse the spread growth in the analysis cycle, simply because this is not required for successfully running assimilation schemes. However, when generating a forecast, reliable growth of ensemble spread demands a detailed analysis of the model instabilities. %
%

%
%
%
%%%%%%%%%%%%%%%%%%%%%%%%%%%%

A popular technique to obtain at least some of this information is the empirical breeding method (\cite{bib:Kalnay93, bib:Kalnay97}). Bred vectors result from an iterative process, which alternates the evolution of a sphere of finite amplitude perturbations with the full non linear equations along a reference trajectory and their regular rescaling. During cycling the sphere of perturbations is distorted along the direction of the fastest growth and must be re-inflated to avoid its degeneration (e.g. \cite{bib:Annan04}). The re-orthogonalisation can be done using the Gram-Schmidt process (e.g. \cite{bib:Benettin80}) or singular value decomposition (SVD) (e.g. \cite{bib:Keller10,bib:Frolov16}). Alternatively logarithmic breeding (\cite{bib:Primo08}) re-inflates the bred vectors by amplifying the small scale variability in space and time. \cite{bib:Wolf85} showed that the long term evolution of the axes of such a directed sphere can be used to estimate the spectrum of the Lyapunov exponents. Since the re-orthogonalisation does not change the first/largest direction of the sphere (bred vector), the corresponding Lyapunov exponent measures the \textit{long term} growth of two nearby trajectories, which stay within a certain subspace of the full phase space, called the system attractor (\cite{bib:Kalney03}). Therefore, the bred vectors usually do not point into the direction of maximal local short term growth, which might be caused by a perturbation outside that subspace.

For a more systematic exploration of the phase space the classical singular vector perturbation (SV) technique has been introduced in NWP by \cite{bib:Mureau} and \cite{bib:Molteni}. The linear propagator of the model, the tangent linear model (TLM), is iterated for a specific optimization time and the fastest growing perturbations are given by the leading (right) singular vectors of the propagator matrix. 
At the European Centre for Medium-Range Weather Forecasts (ECMWF) the singular vector computations are based on a 48 hour time window (see \cite{bib:Leutb_Palmer08}, \cite{bib:IFS_EPS_BARK} and \cite{bib:Magnusson2}). 
They estimate the leading singular vectors by solving an equivalent eigenvalue problem using the Lanczos algorithm (\cite{bib:Lanczos}), which in addition requires the approximation of the adjoint propagator.
\cite{bib:LeutbecherLang}, stated that ''the raison d'\'{e}tre for singular vector perturbations would vanish, if the EDA, together with the representation of model uncertainties, generated enough variance in the space spanned by the leading singular vectors'', but this is still not the case. Computing SV is indeed a great effort, but there is no adequate substitute available, yet. 
For a detailed description about the past findings about SV in atmospheric sciences we refer to \cite{bib:SVReview}.

Recently, \cite{bib:Frolov16} created empirical tangent linear models based on ensembles (ETLM). The main drawback is the rank deficiency of the tangent space spanned by the ensemble perturbations, because the number of ensemble members is generally much smaller than the dynamical dimension. They reduce the rank deficiency problem by localisation assuming that dynamical relations are limited to small finite length scales. This approximation can be justified for data assimilation cycles and short range weather forecasting but does not hold for longer range predictions. 

When using the tangent linear approximation in the classical SV computations, the algorithm might probably miss important non-linear developments. \cite{bib:GilmoreSmith} showed that an ensemble based on these perturbations is limited with respect to the quality of its spread skill relation (see also \cite{bib:Anderson}). To overcome this problem \cite{bib:Mu2000} introduced nonlinear singular vectors (NLSV) and the conditional non-linear optimal perturbations (CNOP) (\cite{bib:Mu03}). Both generalise the SV technique and further research in this field came up by \cite{bib:Mu08a,bib:Mu08b,bib:Mu09}.  \cite{bib:NonLinSV} investigated differences between NLSV and SV in dependence of the amplitude and \cite{bib:DuanHuo16} showed that CNOP is superior to SV in the Lorenz96 system (L96). But for NLSV and CNOP a non-linear optimisation problem has to be solved making these approaches very expensive.

In this paper we introduce another approach for estimating growing perturbations in dynamical systems which
is related to the classical SV technique. We approximate the singular vectors of a specific matrix, which consists of non-linear evolved perturbations of system states, using a block version of the
Arnoldi method (\cite{bib:Arnoldi}). This algorithm is matrix-free and avoids maintaining linearised or adjoint model versions. 
The Arnoldi Algorithm has been used also by \cite{bib:Wei05} in a study on error growth. But instead of computing singular vectors they aim for the leading finite time normal modes (FTNM) of linear propagators. 

We introduce the methodology in section 2 and describe the dynamical systems for tests in section 3. The scores for comparing the growth of different perturbations are given in section 4, the results are presented in section 5, followed by a summary and conclusions (section 6).\\

%% !!!Reihenfolge
\section{Fast growing Perturbations}

\subsection{Validity of the Arnoldi approximation}

For the generation of strongly growing perturbations we make use of the Arnoldi method. In this section we give the basic definitions and describe the assumptions made. Moreover, we discuss the validity of the approximations and provide an estimate of its accuracy. 
We denote the development of a (continuous) dynamical system by the following function
\begin{equation}
\varphi_T: %\C \times
\C^n  \to \C^n, \ x \mapsto \varphi_T(x).
\end{equation}
Here, $T \in \R_+$ is an arbitrary period of time
and $x\in \C^n$ denotes an initial state of the system.
The number of dimensions of the dynamical system is denoted by $n \in \N_+$ and the phase space is $\C^n$. Hence, $\varphi_T(x)$ denotes the evolved state of $x$ after a time $T$. Based on this, we define the evolved increment function (EIF)
\begin{equation}\label{eq:EIF}
	I_{T,x_0,h}: S^n \to \C^n , \ v \mapsto  \varphi_T(x_0+v h) - \varphi_T(x_0),
\end{equation}
where $h \in \R_+$ is the amplitude of a perturbation in direction $v \in S^n$ and the $n$-sphere is denoted by $S^n= \{ x \in \C^n,\  ||x|| = 1 \}$. The  EIF contains the information about the growth of an initial perturbation in $x_0$ after the optimization time interval $T$. 
A perturbation of the state $x_0$ with amplitude $h$ having maximal growth within the optimization interval can be defined as follows:
\begin{equation}\label{eq:MaxGrow}
 \hat{v}_{T,x_0,h} := \underset{v \in S^n}{\operatorname{argmax}} \ ||I_{T,x_0,h}(v)||_N.
\end{equation}

Here, $||\cdot ||_N$ is a norm which measures the perturbation growth. In this paper the Euclidean norm is used for L96 and the Total Energy ``Norm'' for SWM (see section \ref{sec:TEN}). Equation (\ref{eq:MaxGrow}) is known in literature as the conditional non-linear optimal perturbation (CNOP, \cite{bib:Mu03}). We do not aim for a full computation of the CNOP but look for a suitable approximation of growing directions, which allows an efficient estimation of the local expanding subspaces in the models phase space. We introduce the evolved increment matrix (EIM) which consists of (non-linear) evolved perturbations  
\begin{equation}
Y_{T,h}(x_0) :=
\begin{pmatrix}
 I_{T,x_0,h}(e_1)  ,& \ldots
 & , I_{T,x_0,h}(e_n)
\end{pmatrix},
\end{equation}

where $x_0 \in \C^n$ is an (unperturbed) state of the dynamical system and $e_i$ is the $i-$th unit vector. 
The EIM consist of non-linear evolved perturbations. One could use a linear model for obtaining evolved increments as well, but we want to keep as much as possible of the non-linear dynamics. An EIM is \textbf{not} a tangent linear propagator of the system, it is rather a ``collection'' of non-linear (or linear) evolved states, merged in a matrix. 
By using an EIM, one has the ability to work with Krylov subspace methods instead of non-linear optimization solvers. But at the same time, it is not limited to linear growth. 

In particular, if the optimal perturbation $\hat{v}_{T,x_0,h}$ (Equation (\ref{eq:MaxGrow})) is equal to an arbitrary unit vector, the leading singular vector of the corresponding EIM is equal to this optimal perturbation. 
Otherwise a linear combination of non-linear evolved unit vector perturbations can be found. 
The larger the part of an arbitrary unit vector in the linear combination of $\hat{v}_{T,x_0,h}$ is, the smaller becomes the difference between EIM's first singular vector and CNOP.

Assuming a sufficiently smooth trajectory, the error depends on the curvature of the trajectories and on the size of $h$, which can be shown using Taylor approximation. Hence, for sufficiently small quantities the following approximation holds:
\begin{equation}\label{eq:approx_increments}
Y_{T,h}(x_0) v  = \sum_{i=1}^n v_i I_{T,x,h} (e_i) \approx
			 I_{T,x,h}(v) .
\end{equation}
 Here, $v_i$ denotes the $i-$th component of $v$.  

With this approximation, it follows that these matrix-vector-products are approximately equal to an EIF.
This is the legitimation for the changeover from a fully non-linear problem to the EIM in order to enable the use of Krylov subspace methods, such as the Arnoldi Algorithm. In addition, it is needed as an approximation for the upcoming matrix-vector products $Y_{T,h}(x_0) \cdot v$ within the Arnoldi algorithm.
The latter is the major reason for introducing the EIM, a recovery of CNOP is not the main goal here. But it is an important side effect, that the method can also cover non-linear effects, partially.
\\
%

%%%%%%%%

%%%%%%%%%%%%%%%%%%%%%%%%%%5

To obtain the growing directions in the phase space the leading singular vectors of EIM have to be approximated. In this paper we concentrate on the approximation of the first right singular vector of EIM (EIM-SV).

For the computation of singular triplets (singular values and their singular vectors), the Krylov algorithm of Lanczos is widely used. The Lanczos method is applied to solve the equivalent eigenvalue problem with the matrix $Y\TT Y$ which in addition requires the adjoint $Y\TT$. 
In contrast to this, we introduce an approach which uses the Arnoldi method without using the adjoint of the EIM.
Hence, that approximation with the Arnoldi method is rather direct and avoids a detour through the equivalent eigenvalue problem to generate the Krylov subspaces.
 Adjoint-free resp. transpose-free approaches for singular triplet approximation are rare. There is a paper on this topic by \cite{bib:Berry95}.
These authors discussed a block Arnoldi method used for singular value approximation, too but this paper does not contain a theoretical background. We provide a theoretical justification below, in order to legitimate the application of perturbation generation this way. 

Applied to some matrix $B\in\C^{n\times n}$ the block Arnoldi method described in the Algorithm below starts with a (small) number $\ell\in\N$ of orthonormal vectors $q_1,\ldots,q_\ell\in\C^n$ and computes an extended orthogonal basis matrix $Q_r=(q_1,\ldots,q_r)\in\C^{n\times r}$, $r>\ell$ as well as a square matrix $H_r\in\C^{r\times r}$ satisfying
\begin{equation}\label{ArnRes}
 BQ_r=Q_rH_r+G_rE_\ell^T,
\end{equation}
where $G_rE_\ell^T$ is a residual of low rank $\ell$ (\cite{bib:Saad96}).
In detail $E_\ell\in\R^{r\times \ell}$ contains the last $\ell$ columns of the identity matrix and $G_r\in\C^{n\times\ell}$ is orthogonal to $Q_r$, i.e. $Q_r\TT G_r=0$.
The last property shows that $H_r=Q_r\TT BQ_r$ is the projection of $B$ in the subspace $range(Q_r)$.
For $\ell=1$  one obtains the original Arnoldi method.
\par
The basic idea here is to use the singular values of $H_r$ as approximations of those of $B$.
In the ideal case with zero residual $G_r=0$ the identity (\ref{ArnRes}) shows that $Q_r$ is the basis of an invariant subspace of $B$ and the eigenvalues of $H_r$ are also eigenvalues of $B$.
However, this equality no longer holds for the singular values of $H_r$ and $B$.
Hence, the proposed approximation especially makes sense for problems where there is a certain correlation between eigen- and singular values. 
In order to understand accuracy issues the difference between these two sets of singular values will be investigated.

Our estimates are based on a well-known error estimate (e.g. \cite{bib:Parlett80}) for Hermitian matrices $A\in\C^{n\times n}$.
It states that if $(\tilde x,\tilde\lambda)$ is an approximate eigen-pair of $A$ having a small residual $\varepsilon\ge0$, then there exists an eigenvalue $\lambda$ of $A$ within this distance to $\tilde\lambda$:
\begin{equation}\label{everr}
 \|A\tilde x-\tilde\lambda\tilde x\|_2\le\varepsilon\|\tilde x\|_2\ \Rightarrow\ |\lambda-\tilde\lambda|\le\varepsilon.
\end{equation}
This error estimate will be applied to the hermitian matrix
\begin{equation}\label{blkmat}
 A=\begin{pmatrix} 0&B\TT \\B&0 \end{pmatrix}
\end{equation}
which has the singular values of $B$ as eigenvalues.
The error will depend on parts of the matrix $B$ not contained in the projection $Q\TT BQ$. Therefore, an orthogonal extension $W_r\in\C^{n\times(n-r)}$ of $Q_r$ is required such that $(Q_r,W_r)$ is unitary.
\begin{theorem}\label{Theorem1}
Let (\ref{ArnRes}) hold for $B\in\C^{n\times n}$ and $H_r=U\Sigma V\TT$ be the SVD of $H_r$ with $\Sigma=diag(\sigma_1,\ldots,\sigma_r)$.
Then, for each $i\in\{1,\ldots,r\}$ there exists a singular value $\hat\sigma$ of $B$ satisfying the following estimate
\begin{equation}\label{sigerra}
% |\hat\sigma-\sigma_i|^2\le\frac12(\|Q_r\TT BW_r\|_2^2+\|G_r\|_2^2).
 |\hat\sigma-\sigma_i|\le \frac{1}{\sqrt{2}} \sqrt{ \|Q_r\TT BW_r\|_2^2+\|G_r\|_2^2 }.
\end{equation}
\end{theorem}
\begin{proof}
With the singular vectors $u_i,v_i$ the transformed vectors $Q_ru_i$, $Q_rv_i$ are approximations to singular vectors of $B$ since
 $\begin{pmatrix}
 Q_r v_i\\
 Q_r u_i
 \end{pmatrix},\,i\in \{ 1,\ldots,r \}$ are approximations to eigenvectors of $A$ from \eqref{blkmat}.
 With the decomposition $I=Q_rQ_r\TT+W_rW_r\TT$ it holds due to (\ref{ArnRes}) that
\begin{eqnarray*}
 (A-\sigma_i I)\begin{pmatrix} Q_rv_i\\ Q_ru_i \end{pmatrix}
  &=&\begin{pmatrix}
   Q_rQ_r\TT B\TT Q_ru_i-\sigma_iQ_rv_i+W_rW_r\TT B\TT Q_ru_i\\
   BQ_rv_i-\sigma_iQ_ru_i
  \end{pmatrix}\\
 &=& \begin{pmatrix}
   W_r(Q_r\TT B W_r)\TT u_i\\
   G_r E_\ell^T v_i
  \end{pmatrix}.
\end{eqnarray*}
Since $\|Q_rv_i\|_2^2+\|Q_ru_i\|_2^2=2$ the error estimate (\ref{everr}) follows with $2\varepsilon^2=\|W_r(Q_r\TT B W_r)\TT u_i\|_2^2+\|G_r E_r^T v_i\|_2^2\le\|Q_r\TT B W_r\|_2^2+\|G_r\|_2^2$ due to orthogonality $Q_r\TT G_r=0$.
$ \blacksquare$
\end{proof}
Some discussion is required on the meaning of estimate (\ref{sigerra}).
The residual $\|G_r\|_2$ is available during execution of the Arnoldi algorithm and its size may be used as a stopping criterion for choosing the dimension $r$ of the basis.
So, it may be quite small and for the moment we neglect it in order to concentrate on the meaning of the norm $\|Q_r\TT BW_r\|_2$.% in (\ref{sigerra}).
\par
This norm is related to the non-normality of the matrix $B$.
Normal matrices commute with their adjoint, $BB\TT=B\TT B$ and the triangular matrix $R_B = \Lambda_B+N_B$ in the Schur normal form $ B = S_B R_B S_B^*$ with unitary $S_B$ is diagonal, $N_B=0$.

Hence, if $\|BB^* - B^*B\|$ or $\|N_B\|$ becomes small relative to $\| B \|_2$ the errors become small.

A convenient way to describe non-normality is its departure (\cite{bib:Henrici62})
\begin{align}\label{devnrm}
 \dep_F(B):=\|N_B\|_F=\sqrt{\|B\|_F^2-\|\Lambda_B\|_F^2}.
\end{align}

Using the departure gives the possibility to derive another more concrete variant of the estimate.
In these discussions often the unitarily invariant Frobenius matrix norm $\|B\|_F^2:=\sum_{i,j=1}^n|b_{ij}|^2$ is used in literature.  Unfortunately, the Frobenius norm may lead to weaker estimates in many cases.
\begin{lemma}
Let $Q\in\C^{n\times r}$ be orthonormal and $W$ a basis extension such that $(Q,W)$ is unitary.
If $BQ=QH$, then
\begin{align}\label{cschranke}
 \|Q\TT BW\|_2\le \|Q\TT BW\|_F\le \sqrt{\dep_F(B)^2-\dep_F(H)^2}.
\end{align}
\end{lemma}
\begin{proof}
Again with $QQ\TT+WW\TT=I$ holds
\begin{align*}
 \dep_F (B)^2 =& \| B(QQ\TT +WW\TT ) \|_F^2- \| \Lambda_B \|_F^2 \\
  =& \|QH\|_F^2+\|BW\|_F^2-\| \Lambda_B \|_F^2 \\
  =& \| H \|_F^2 + \|Q\TT BW \|_F^2+ \|W\TT BW\|_F^2-\|\Lambda_B\|_F^2.
\end{align*}
Now, since $Q$ defines an invariant subspace of $B$, the eigenvalues in $\Lambda_B$ consist of those from the diagonal part $\Lambda_H$ of $H$ and the diagonal part $\tilde\Lambda$ of $W\TT BW$.
Hence, rearranging the terms gives
\begin{align*}
\|Q\TT BW\|_F^2=\dep_F(B)^2-\dep_F(H)^2-\dep_F(W\TT BW)^2,
\end{align*}
from which (\ref{cschranke}) follows. $\blacksquare$
\end{proof}

Hence, with (\ref{sigerra}) and for $G_r=0$ the error estimate for the singular values becomes
\begin{equation}
 |\hat{\sigma}(B)-\sigma_i(H_r)|\le\frac1{\sqrt{2}}\sqrt{\dep_F(B)^2-\dep_F(H_r)^2}. 
\end{equation}
This shows that the error will be small if $B$ is not too far from normality and/or if the Krylov subspace spanned by $Q_r$ is large enough to collect sufficient amount of information about $B$ in the projection $H_r=Q_r^\ast BQ_r$.
\par
In time-dependent partial differential equations often the part with the highest order is linear and its discretisation leads to a hermitian or skew-hermitian matrix depending on the order.
Lower order terms correspond to compact perturbations, only.
In these cases $B$ may be nearly normal according to the following estimates due to \cite{bib:LeeSL95}.
It uses the hermitian ${\cal H}(B):=\frac12(B+B\TT)$ and skew-hermitian ${\cal S}(B):=\frac1{2i}(B-B\TT)$ parts of matrix $B$ and states that
\begin{align}\label{Lee}
 \dep_F(B)^2= 
  2(\|{\cal H}(B)\|_F^2-\|Re(\Lambda_B)\|_F^2)=
  2(\|{\cal S}(B)\|_F^2-\|Im(\Lambda_B)\|_F^2).
\end{align}
So, near-normality of $B$ is given, if its skew-hermitian or its hermitian part is small compared to $\|B\|$ leading to small errors in singular values.\\
 The ratio $\|BB^* - B^*B\|_2 \ /\ \|B\|_2$  is favourable in both systems ($< 0.05$ in L96 and $< 0.1$ in SWM). Unfortunately, Formula (\ref{cschranke}) provides only a weak estimate for both systems. Anyway, computations showed that e.g. twelve dimensional subspaces for L96 (with 50 dimensions) achieve a ratio of $\|H\|_2 / \|B\|_2 >0.8$. Almost the same holds for 40 dimensional subspaces in SWM (1587 dimensions).

\subsection{Block Arnoldi Algorithm}

The Arnoldi method is matrix-free and does not require knowledge of the whole EIM. 

As stated above, the ability to approximate matrix-vector-products $Y_{T,h}(x_0) \cdot w$ of the EIM with arbitrary vectors is needed, only.
This can be done by using the approximation (\ref{eq:approx_increments}). %and theorem (\ref{theo:linearEIM}) supply the possibility 
Hence, although the EIM is the central matrix for the 
%%rot ende
singular vector approximation, the matrix is not used practically in the algorithm described below.

The Arnoldi method simultaneously generates a sum of Krylov subspaces with an orthonormal basis $Q_r$ and also a generalized Hessenberg matrix $H_r$, which is the best representation of $B$ within that subspace.

In this article a block extension of the classical Arnoldi algorithm is used, (\cite{bib:block-arn}). The block method allows to start the iteration with $\ell \in \N_+$ linearly independent vectors. 
The block size within the iteration process is implicitly specified by the number $\ell$ of starting vectors, so these two values are identical. The dimension of the generated subspace equals the number of iteration loops times the block size.
We call the perturbations generated this way block Arnoldi perturbations (BAP). The procedure to compute BAP, is described in Algorithm BAP below.

The block in the $i-$th row and the $j-$th column of the block Hessenberg matrix $H$ is denoted with $H_{[i,j]} \in \C^{\ell \times\ell}$.
Furthermore, $(\cdot)_{[i]}$ is a column vector, which denotes the $i$-th column of the corresponding matrix and $Q=(Q^{(1)},\ldots ,Q^{(m)}) \in \C^{n \times\ell m}$ contains the orthonormal basis of the generated Krylov subspace.\\[0mm]

\textbf{Algorithm BAP}\\[-8mm]
%\begin{algorithm}\label{algor1}
\begin{algorithmic}[1]\Statex
\Procedure{BAP}{$T,h,x_0,Q^{(1)},m,k$}
\Statex $T \in [0,\infty)$ \Comment{optimization time}
\Statex $h \in [0,\infty)$ \Comment{amplitude of perturbation}
\Statex $x_0 \in \C^n$ \Comment{unperturbed state of the system}
\Statex $Q^{(1)} \in \C^{n\times\ell}$\Comment{$\ell$ initial vectors merged in a matrix}
\Statex $m \in \N$\Comment{number of loops within Arnoldi algorithm}
\Statex $k \leq m\ell\in \N$\Comment{number of BAP to be returned}
\Statex

\State $Q^{(1)} \gets \text{orth}(Q^{(1)})$ \Comment{orthonormalization}

\For{$j=1 :  m$}
\For{$i=1:\ell$}

\State $W_{[i]} \gets I_{T,x_0,h} \left( Q_{[i]}^{(j)} \right) $\Comment{$W \in \C^{n \times\ell}$ }
\EndFor

\For{$i=1:j$} \Comment{Gram-Schmidt orthogonalization}
\State $H_{[i,j]} \gets Q^{(i)\, T\,} W$

\State $W \gets W-Q^{(i)}H_{[i,j]}$
\EndFor

\State $[Q^{(j+1)} ; H_{[j+1,j]}] \gets \text{qr}(W)$\Comment{QR-decomposition}
\EndFor

\State $[U ; \Sigma ; V ] \gets \text{svd} \left( H \right)$ \Comment{singular value decomposition}
\For{$i=1:k$}
\State $P_{[i]} \gets Q V_{[i]}$  \Comment{$Q = (Q^{(1)}, \ldots , Q^{(m)}) \in \C^{n\times\ell m}$}

\EndFor

\State \textbf{return} $P$

\EndProcedure

\end{algorithmic}
%\end{algorithm}
\textbf{end}\\

The first $r$ vectors of the basis matrix $Q$ are denoted with $Q_r$ (suitable for $H_r$ and $H$). These are computed in the algorithm up to line 12 and satisfy Equation (\ref{ArnRes}).
The final transformation leads with $H_r = U \Sigma V^*$ to an approximate singular vector relation 
\begin{equation}
B(Q_r V) = (Q_r U) \Sigma + G_r E_r^T V.
\end{equation}

The matrix $P\in\C^{n\times k}$ contains the first $k$ right singular vectors of $Y_{T,h}(x_0)$, which are the first $k$ columns of $Q_r V$. For this paper
always $k=1$ is used. Hence, BAP refers here to the approximation of the first singular vector only.

This routine allows the generation of initial perturbations for strongly (non-linear) growth. There is no need for an expensive classical optimization solver nor the usage of an equivalent eigenvalue problem. It allows to compute these approximations without linear or adjoint models of the system.

The cost for BAP (in complex systems) are mainly determined by the number of required model integrations.
This can be calculated with the following equation:
\begin{equation}
IG_{\text{BAP}}(l,m,T,\Delta t) = ml \ceil[\bigg]{\left( \frac{T}{\Delta t} \right) }. 
\end{equation}
where $\Delta t$ denotes the discrete time step of the model and $\ceil[\big]{\cdot}$ denotes the ceiling function. It may be of great advantage, that each Arnoldi loop can be divided into $l$ parallel computations.

\subsection{Starting Arnoldi}

The choice of the initial vectors for the block Arnoldi iteration is relevant for a good performance. We test two different strategies of choosing initial vectors: random initialisation and chord vectors as differences between states on the past trajectory (see Figure \ref{fig:chord}). Random vector starts provide a benchmark to compare with. The choice of chord vectors may be varied by choosing length and distance between two consecutive chords. For the experiments in L96 and SWM we used \textit{two} discrete time steps from start to end of one chord and skip 13 steps between two of them. 

We assume that directions obtained from chord vectors are more closely related to the local dynamics and, because they are derived from past states, the chords represent (almost) balanced directions. For shorter periods of time trajectories may move in a relatively small subspace of the complete phase space. Hence, a span of chord vectors should lead to a richer subspace and balanced which is more relevant for the local dynamics and may provide a better start up for the Arnoldi iteration than a random selection of perturbations. Nevertheless, the first chord vector alone will probably not be a good initial choice because it is almost a tangent of the trajectory.  

Differences between states separated by a relativ long period of time have already been used to generate perturbations in dynamical systems. For example, \cite{bib:Rand_Field} as well as \cite{bib:Mag_Rand_Field} used differences between weather patterns from the past as ``Random Field'' perturbations to initialize ensemble weather forecasts.

\begin{figure}[h]
\centering
  \includegraphics[trim = 0mm 0mm 0mm 0mm, clip,width=12cm]{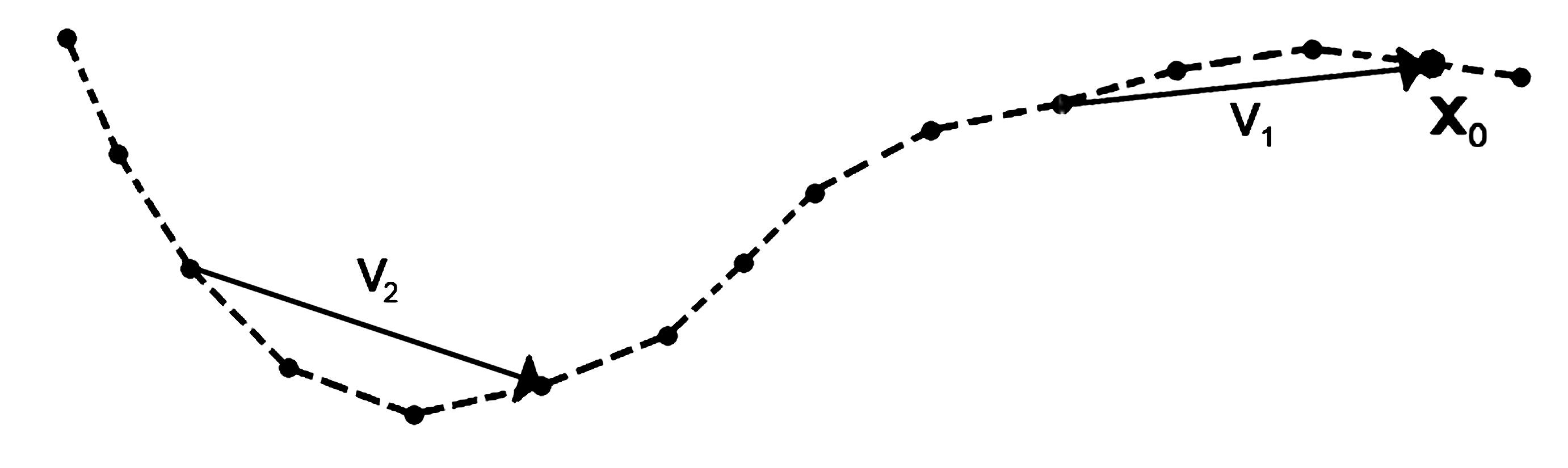}
  \caption[Bildunterschrift]{Illustration of chord vectors. Dashed line: trajectory, arrows: first and second chord vectors $v_1$ and $v_2$ of the state $x_0$}\label{fig:chord}
\end{figure}

\section{Dynamical systems}

We test the BAP technique for two different dynamical systems. These are L96 (\cite{bib:l96-paper}) and a shallow water model (SWM, e.g. \cite{bib:SWEq}) on a two dimensional domain. L96 is derived from fluid convection and is based on a system of ordinary differential equations with a scalable number of degrees of freedom. It is defined by
\begin{eqnarray} \label{lorenz96}
\frac{d}{dt}y_i = -y_{i-2}y_{i-1}+y_{i-1}y_{i+1}-y_i+F,
\end{eqnarray}
where $y_i$ is the $i$-th component of $y=y(t) \in \R^K$, $i \in \{1, ... , K\}$ and $y_{-1} := y_{K-1} , y_{0}:= y_K$ and $y_{K+1}:=y_{1}$. The number of dimensions is given by $K>3 \in \N $ and  $F \in \R_+$ is a constant value, the so called ``Forcing Term''. In this paper $F=8$ and $K=50$ is used, which is similar to the choice Lorenz originally discussed ($F=8,K=36$) and we start the integrations perturbing the fixed point, $\hat{y}=(F,\dots,F)$. Under the described conditions L96 shows chaotic behaviour. For the computation the standard fourth order Runge Kutta method is used with a discrete time step of $0.01$.

In contrast to L96 the SWM is based on hyperbolic partial differential equations. It describes the evolution of waves under the condition, that the horizontal length scales are much larger than the vertical one. The basic form of this model is given by
\begin{eqnarray}\label{SWM2D}
\frac{\partial h}{\partial t} + \frac{\partial (uh)}{\partial z_1}+ \frac{\partial (vh)}{\partial z_2} = 0 \nonumber  \\
\frac{\partial (uh)}{\partial t} + \frac{\partial (uh)}{\partial x_1}u+
 \frac{1}{2} g \frac{\partial}{\partial z_1} h^2 + \frac{\partial (vh)}{\partial z_2}u = 0\\
\frac{\partial (vh)}{\partial t} + \frac{\partial (vh)}{\partial z_1}u+ \frac{\partial (vh)}{\partial z_2}v +
\frac{1}{2} g \frac{\partial}{\partial z_2}  h^2 = 0 \nonumber
\end{eqnarray}
Here, $t \in [0,\infty)$ is the time variable, $z=(z_1,z_2)^T  \in \Omega$ the two-dimensional, horizontal space variable and $g$ is the acceleration due to gravity, which is $g=9.81$. The functions $h=h(t,z)$, $u=u(t,z)$,  $v=v(t,z)$ describe the height of the fluid above ground, the speed of the fluid in direction $z_1$ and $z_2$ respectively. The first equation ensures mass conservation, the two others impulse conservation. The boundary conditions are of Dirichlet type, which means that
$u(t,z)=0, v(t,z)=0 \ \text{for all} \ z\in \partial \Omega $, which causes reflectivity at the domains boundary. 
In this paper the SWM is solved on the two dimensional rectangular domain $\Omega = [0,22]^2 \subset \R^2$. We use a discrete time step of $0.01$ and a mesh size for the domain of $1$, which leads to $23^2=529$ grid points ($441$ inner plus $88$ boundary points).
Together with the three physical variables (height of the fluid, velocity in both horizontal directions) at each grid point, we obtain a dynamical system with $1587$ degrees of freedom in total. The hyperbolic initial-boundary-problem is solved with the finite difference scheme of Lax-Wendroff (\cite{bib:lax-wend-orig}). \cite{bib:SWM_Saiduzz} showed that the Lax-Wendroff scheme is adequate for solving the SWM.

\subsection{Total energy in SWM}\label{sec:TEN}

The SWM is a model of a physical system. The chosen norm for perturbation generation should take this into account and a measurement based on physical properties is more convenient here than the Euclidean norm. 
 The total energy of a state $x = (h,u,v)$ in SWM can be computed with 
\begin{equation}
E(x) =\frac{1}{2} \int_\Omega h (u^2 + v^2) + g h^2 d\Omega   
\end{equation} 
(see \cite{bib:SWMTEN}). Obviously, $E(x)$ is not a true norm since it lacks homogeneity. However, the name ``Total Energy Norm'' can be found quite often. 

In order to remove the problem, that $E$ is not a true norm, transformed states are introduced:  
\begin{equation}\label{eqn:transformed}
\tilde{x} := (\tilde{h}, \tilde{u}, \tilde{v}) := \left(\sqrt{\frac{g}{2}} h, \frac{1}{2}\sqrt{h} u, \frac{1}{2}\sqrt{h} v \right).
\end{equation}
 Then the equality
$
||\tilde{x}||_2^2 = E(x)
$
holds. With these modifications, the Euclidean norm can be used for energy measurement and also the standard scalar product is available. 

 We run the SWM also with the euclidean norm and obtained qualitatively similar results, but do not discuss these here.

\section{Measuring perturbation growth}\label{sect:growth}

The comparison of the different perturbation methods with respect to their ability of generating growth is mainly done by computation of short term growth rates within a limited period of time. The logarithm of the growth rates are used, since we work with exponential growth rates (EGR) which are defined as follows:
\begin{equation}
\text{EGR}_{x_0,\text{PT},\Delta t} :\R^+ \to \R , \ t \mapsto
 \frac{1}{\Delta t}\left(\log \left( \frac{||\varphi_t(x_0+v) - \varphi_t(x_0)||}{||\varphi_{t-\Delta t}(x_0+v) - \varphi_{t-\Delta t}(x_0)||} \right) \right).
\end{equation}
Here, $x_0$ is an unperturbed state of the system and $v \in \C^n$ is a perturbation defined by the perturbation technique PT. The discrete time step $\Delta t$ is set equal to the discrete time step of $0.01$. The amplitude $h$ (see Equation~(\ref{eq:EIF})) of  perturbations is $h=0.015$ in L96 and $h=0.035$ in SWM throughout this paper. 
The EGR of a specific perturbation technique will vary at different states of the system. Hence, we work with the arithmetic mean of several EGR in order to obtain a mean exponential growth rate (MEGR) (cf.                                                                                                                                                                                                                                                                                                                                                                                                                                                                                                                                                                                                                                                                                                                                                                                                                                                                                                                                                                                                                                                                                                                                                                                                                                                                                                   \cite{bib:Magnusson-diss}). The presented MEGR for different perturbation techniques are always computed from a mean of $100$ EGR, which are obtained at randomly chosen states along the reference trajectory.
In order to ensure that all transients have decayed, at least $1500$ time steps were computed before the perturbation experiments start.

 Besides the MEGR, for comparison of the growth for different techniques, an integral of the MEGR is used as value of reference. To obtain a single scalar value, we approximate the integral within the optimization time and relate it to the corresponding value of EIM-SV. EIM-SV perturbations are the basis for BAP thus a comparison of both is naturally and reasonable.
This leads us to the definition of the relative exponential growth integral (REGI):
\begin{equation}\label{REGI}
			REGI_{T,\text{PT},\Delta t} := 
			\frac{  \int_0^{T} MEGR_{\text{PT},\Delta t}(t) \ dt  }{\int_0^{T} MEGR_{\text{EIM-SV},\Delta t}(t) \ dt },
\end{equation}
where the growth rates are aggregated within the optimization time $T=0.2$. 

 Please note that the construction of EIM-SV and BAP perturbations is only based on states which occur at the end of the optimization period. The exact behaviour before and the whole behaviour after this period is not taken into account. In contrast to this, the MEGR and REGI are based on momentary growth rates. This is beneficial due to additional information but it does not fit exactly the way the optimization of perturbation growth is done.

For SWM also the computational costs of BAP and the full EIM-SV are compared, relatively.  
A detailed comparison of the costs is of less interest for L96, due to its smaller size.

\section{Results}

\subsection{Lorenz 96}

In L96 EIM-SV perturbations lead to growth curves with a strong exponential growth in the beginning,
which decay relatively fast later on. 
Figure \ref{fig:L96_GR} compares the mean exponential growth rates of BAP with those from the full EIM-SV, where the BAP approximations have different numbers of iterations. Figure \ref{fig:L96_GR}, (a)-(c)  present the BAP experiments with randomly selected initial vectors (BAP-R) while Figure \ref{fig:L96_GR}, (d)-(f) shows the SV approximations starting from chord vectors (BAP-C). The REGI measurement for both BAP-R and BAP-C can be found in Table \ref{tab:L96}.
The short-term growth peak within the optimisation time and converge to a certain level of average error growth for longer lead times. This convergence level is an intrinsic property of the dynamical system and is determined by the largest Lyapunov exponent (see \cite{bib:Lyapunov_original, bib:Lyapunov_1992}). 

BAP obtained from one iteration loop (in Algorithm BAP $m=1$) do not involve a (real) Arnoldi step. In this case the algorithm simply generates the best directions within in the subspace spanned by the initial vectors. Consequently, in Figure \ref{fig:L96_GR} (a) it is seen that one random perturbation does not show positive growth but a
shrinking behavior in the beginning. 

\begin{figure}[h]
\centering
  \includegraphics[trim = 0mm 0mm 0mm 0mm, clip,width=15.5cm]{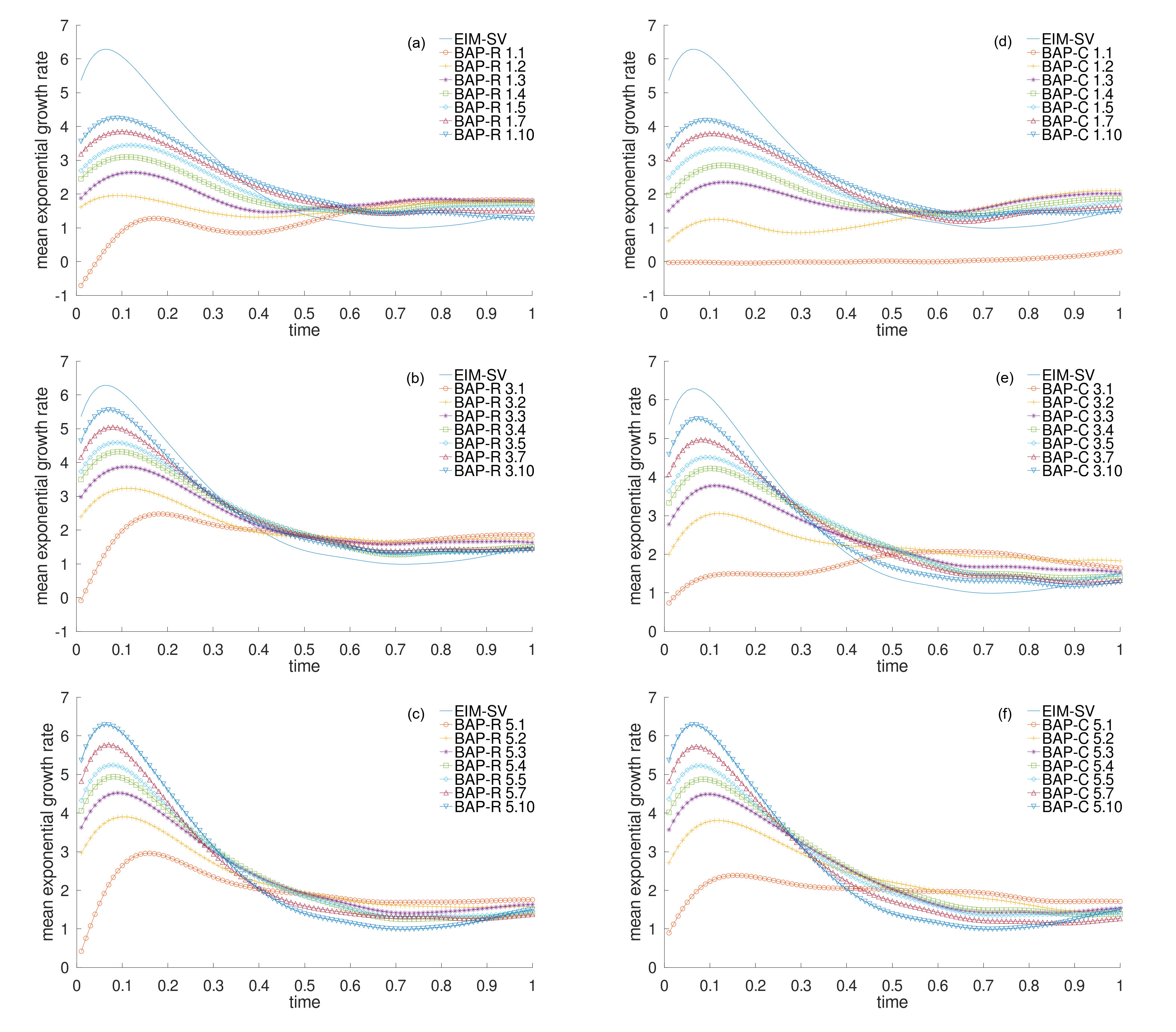}
  \caption[Bildunterschrift]{MEGR for L96 - EIM-SV and \underline{r}andom vector started BAP (BAP-R) (left hand side) or \underline{c}hord vector started BAP (BAP-C) (right hand side). BAP-R/C i.j denotes an approximation of EIM-SV obtained by the usage of i initial vectors and j iteration loops for BAP. Plot (a) BAP-R $i=1$, (b) BAP-R $i=3$, (c) BAP-R $i=5$, (d) BAP-C $i=1$, (e) BAP-C $i=3$ and (f) BAP-C $i=5$. }\label{fig:L96_GR}
\end{figure}

\clearpage

\begin{table}[bt]
\caption{L96 - BAP - REGI with $T=0.2$ (see \autoref{REGI})} \label{tab:L96}
\centering
\begin{threeparttable}
%\begin{tabular}{lrrrrrp{0.1cm}rrrrr}
\begin{tabular}{lrrrrrrrrrrr}
\headrow
\multicolumn{6}{r}{\textbf{number of random initial vectors}} & \multicolumn{6}{r}{\textbf{number of initial chord vectors}}\\
\headrow
%\thead{iterations} & \thead{1} & \thead{2} & \thead{3} & \thead{4} & \thead{5} & & \thead{1} & \thead{2} & \thead{3} & \thead{4} & \thead{5}\\
\textbf{iterations} & \textbf{1} & \textbf{2} & \textbf{3} & \textbf{4} & \textbf{5} & & \textbf{1} & \textbf{2} & \textbf{3} & \textbf{4} & \textbf{5}\\
\hline
%\hiderowcolors
1&  0.12 &  0.23 &  0.30 &  0.36 &  0.40  & &  0.00 &  0.10 &  0.23 &  0.30 &  0.36 \\
2&   0.33 &  0.46 &  0.53 &  0.59 &  0.64 & & 0.19 &  0.40 &  0.50 &  0.57 &  0.63 \\
3&   0.43 &  0.55 &  0.64 &  0.70 &  0.75  & & 0.38 &  0.52 &  0.62 &  0.69 &  0.75 \\
4&    0.52 &  0.64 &  0.72 &  0.77 &  0.81 & & 0.46 &  0.61 &  0.70 &  0.76 &  0.81 \\
5&   0.57 &  0.69 &  0.76 &  0.81 &  0.85  & &  0.55 &  0.67 &  0.75 &  0.81 &  0.85 \\
7&    0.64 &  0.76 &  0.82 &  0.88 &  0.92 & &  0.63 &  0.74 &  0.82 &  0.87 &  0.92 \\
10&   0.71 &  0.82 &  0.89 &  0.95 &  1.00 & &  0.69 &  0.80 &  0.89 &  0.95 &  1.00 \\
\end{tabular}
\end{threeparttable}
\end{table}

Because of its construction, the first chord vector
is almost tangential to the trajectory. Hence, perturbations pointing in this direction show almost no growth in the beginning (Figure \ref{fig:L96_GR} (d)). 
This disadvantage is motivation for using more chord
vectors and it vanishes, indeed, if this is done (Figure \ref{fig:L96_GR}, (e) (f)). BAP performance in L96 is very similar for chord and random initialisation (Table \ref{tab:L96}).

%
%By using five initial vectors and seven iteration loops BAP reaches a level of growth which is quite close to the
%one of EIM-SV. This is not surprising since the resulting Krylov subspaces have dimension 35, which is
%quite close to the full system dimension (50). 
One can see that the EIM-SV perturbations are fully recovered if and only if the size of the Krylov space reaches the size of the system. Krylov spaces of smaller size should cover growing directions above average, but one would expect that at least a small part of the EIM-SV is not covered therein.
It is of interest that large parts of the growth are covered already in relatively small subspaces, which is the case here.

The computational cost-advantage of BAP relative to EIM-SV is relatively small for this problem, because Krylov subspace methods are
much more efficient in larger systems. The dimensional size of L96 is still small for Krylov methods. Using subspaces with e.g. 25 dimensions leads to a REGI of 0.85, in exchange an effort of 0.5 is needed, relative to the full computation. 
For the full recovery of EIM-SV with BAP, through the usage of Krylov spaces with dimension 50, the same computational effort is required (even a bit more), which is expectable a priori.
The choice of initial
vectors is irrelevant for the computational costs since these are determined by the size of the subspaces only.
Indeed, results in L96 are less interesting according to computational costs, but in this system one can observe that BAP is functional in principle. 
To investigate the behaviour of BAP in larger systems the SWM is considered.

\subsection{Shallow Water Model}

In SWM the ratio of Krylov subspace sizes can be much smaller compared to the system dimension
(1587). Due to the larger dimension of SWM, we extended the range for testing up to a Krylov subspace size of 1000. 

One should note that the absolute growth rates depend on the chosen norms and intrinsic properties of the used model. Hence, these should not be compared directly. 

In the same way as for L96 Figure \ref{fig:SWM_GR} shows the MEGR for the SWM. The
growth curves of EIM-SV and BAP are qualitatively similar in L96 and SWM. This includes a relatively
fast decay of growth rates, the shape of the curves and the behaviour of single chord vectors or
random vectors.

In SWM EIM-SV produces damped oscillations of MEGR after the period of optimization. With an increasing block size and a larger number of iterations BAP can reproduce the oscillating character of EIM-SV. Once again the EIM can not cover anything after this period, but this behaviour is interesting according intrinsic properties of SWM. 

In fact, in SWM it is obvious that the quality of the BAP approximation is not only determined by the dimension of the Krylov subspace but also depends on the number and choice of initial vectors.

It can be seen in Figures \ref{fig:SWM_GR} and Table \ref{tab:SWM_REGI} that the increase of growth rates is in many cases smaller if random initial vectors were chosen. Using chord vectors enables for many cases a better approximation especially in smaller subspaces.  

Perturbations obtained from a e.g. $40$-dimensional Krylov subspace are already sufficient to almost accurately reproduce even the oscillating character of 
the leading EIM-SV in Figure \ref{fig:SWM_GR}. Although the chord choice for initial vectors shows a curious behaviour in SWM (with total energy ``norm''): Sometimes the REGI measurements shrink although the size of the subspaces increases (Table \ref{tab:SWM_REGI}).

\begin{table}[bt]
\caption{SWM - BAP - REGI with $T=0.2$ (see (\autoref{REGI})) }\label{tab:SWM_REGI}
\centering
\begin{threeparttable}
\begin{tabular}{lrrrrrrrrrrrrr}
\headrow
\multicolumn{7}{r}{\textbf{number of random initial vectors}} & \multicolumn{7}{r}{\textbf{number of initial chord vectors}}\\
\headrow
\textbf{iter.} & \textbf{1} & \textbf{2}  & \textbf{4} & \textbf{5} &  \textbf{10} & \textbf{20} & & \textbf{1} & \textbf{2}  & \textbf{4} & \textbf{5} &  \textbf{10} & \textbf{20} \\
\hline
%\hiderowcolors
1 &  0.00 &  0.03 &    0.06 &  0.06 &  0.09 &  0.12 &  &  0.00 &  0.11 &    0.23 &  0.26 &  0.38 &  0.53 \\
2 &  0.09 &  0.09 &    0.11 &  0.12 &  0.15 &  0.18 &  &  0.14 &  0.43 &   0.55 &  0.59 &  0.67 &  0.71 \\
3 &  0.19 &  0.17 &    0.19 &  0.20 &  0.23 &  0.28 &  &  0.21 &  0.49 &   0.73 &  0.75 &  0.77 &  0.75 \\
4 &  0.24 &  0.23 &    0.26 &  0.28 &  0.32 &  0.41 &   &  0.27 &  0.52 &    0.73 &  0.74 &  0.73 &  0.68 \\
5 &  0.28 &  0.29 &   0.33 &  0.35 &  0.43 &  0.55 &  &  0.30 &  0.53 &    0.71 &  0.71 &  0.68 &  0.65 \\
10 &  0.28 &  0.34 &  0.44 &  0.48 &  0.63 &  0.79 &   &  0.36 &  0.50 &    0.69 &  0.70 &  0.75 &  0.79 \\
20 &  0.33 &  0.42 &   0.58 &  0.64 &  0.81 &  0.90 &   &  0.39 &  0.57 &    0.72 &  0.74 &  0.79 &  0.82 \\
50 &  0.52 &  0.68 &   0.80 &  0.83 &  0.89 &  0.93 &   &  0.55 &  0.68 &    0.79 &  0.81 &  0.86 &  0.90 \\
\end{tabular}
\end{threeparttable}
\end{table}

As stated in section \ref{sect:growth}, this is possible due to the construction of the perturbations and the definition of the growth 
scores, although this is an unwanted but noticeable behaviour. Because it does not occur if random initial vectors were chosen, unfavourable properties of the chosen chords in this system may cause this. 
However, the reasons causing this are not clear. 

Nevertheless, chord vectors are an interesting and still promising option for initial vectors here. At least we would like to point out, that the results show, that the choice of initial vectors is relevant for the method.\\

Besides the behaviour of the growth, the spatial structure of an EIM-SV and two approximations thereof are shown in Figure \ref{fig:sp_struct}. 
The EIM-SV (Figure \ref{fig:sp_struct}, first row) has weight concentrates at the inner grid points of the top left corner.

Perturbing the corners is reasonable, because this will cause interfering waves after reflection at the boundaries. Hence, we observed that in SWM EIM-SV perturbations occur at the four corners of the domain. 
This structure is quite well recovered by a BAP (started with chord vectors) obtained from a subspace of dimension 1000 (Figure \ref{fig:sp_struct}, third row). Also the BAP computed in a subspace of size 40 (Figure \ref{fig:sp_struct}, second row) contains a relevant weight at the top left corner, but of course the overall picture is less clear. However, taking into account that the used subspace here covers less than three percent of the whole space, this can be classified as a good result.   
\\

\begin{table}[ht]
 \caption{SWM - BAP - computation time relative to EIM-SV } \label{tab:SWM_time}
\centering
\begin{threeparttable}
\begin{tabular}{lrrrrrr}
\headrow
 \multicolumn{7}{r}{\textbf{number of initial vectors}}\\
\headrow
%\thead{iterations} & \thead{1} & \thead{2} & \thead{4} & \thead{5} & \thead{10} & \thead{20}\\
\textbf{iterations} & \textbf{1} & \textbf{2} & \textbf{4} & \textbf{5} & \textbf{10} & \textbf{20}\\

\hline
%\hiderowcolors
  1  & 0.001  & 0.001    & 0.002  & 0.003  & 0.006  & 0.012 \\
  2  & 0.001  & 0.002    & 0.005  & 0.006  & 0.012  & 0.025 \\
 3 & 0.002  & 0.004    & 0.007  & 0.009  & 0.019  & 0.037 \\
  4 & 0.002  & 0.005    & 0.010  & 0.012  & 0.025  & 0.050 \\
  5  & 0.003  & 0.006   & 0.012  & 0.015  & 0.031  & 0.062 \\
  10  & 0.006  & 0.012    & 0.024  & 0.031  & 0.062  & 0.124 \\
 20  & 0.012  & 0.024   & 0.049  & 0.061  & 0.123  & 0.247 \\
  50  & 0.031  & 0.061    & 0.122  & 0.153  & 0.306  & 0.620 \\
\end{tabular}
%\begin{tablenotes}
%\item JKL, just keep laughing; MN, merry noise.
%\end{tablenotes}
\end{threeparttable}
\end{table}

\clearpage

\begin{figure}[h]
\centering
  \includegraphics[trim = 0mm 0mm 0mm 0mm, clip,width=15.5cm]{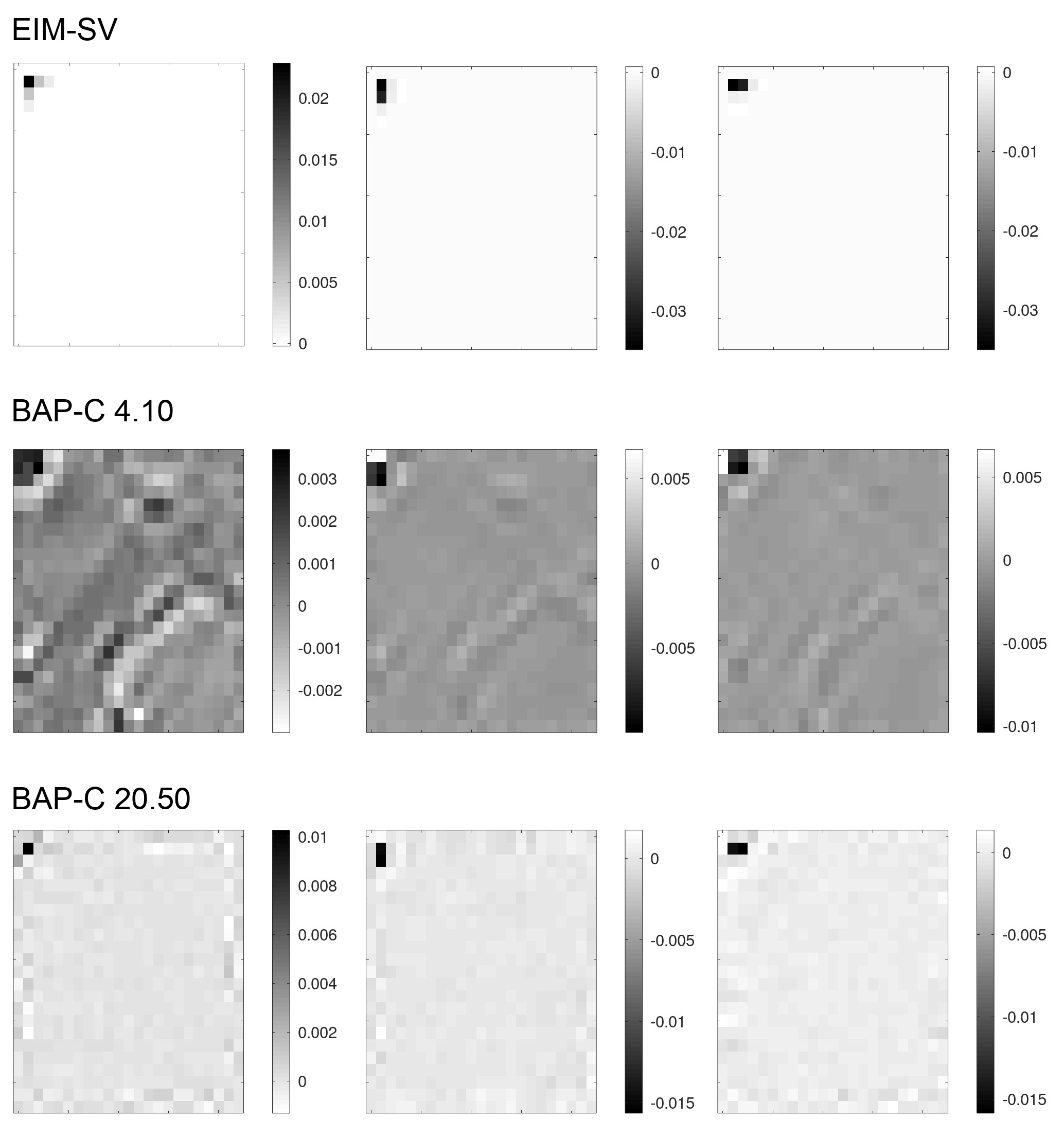}
  \caption[Bildunterschrift]{ Spatial structure of an EIM-SV perturbation and two approximations thereof using BAP (started with chord vectors) in SWM. The squares display the grid points of the domain $\Omega$.  
  Here, 4 (20) initial vectors and 10 (50) iteration loops are used for approximation. The first (second / third) column shows the $\tilde{h}$ ($\tilde{u}$ / $\tilde{v}$) (see. Eq(\ref{eqn:transformed})) part for each perturbation, respectively. 
  }\label{fig:sp_struct} 
\end{figure}

\begin{figure}[h]
\centering
  \includegraphics[trim = 0mm 0mm 0mm 0mm, clip,width=15.5cm]{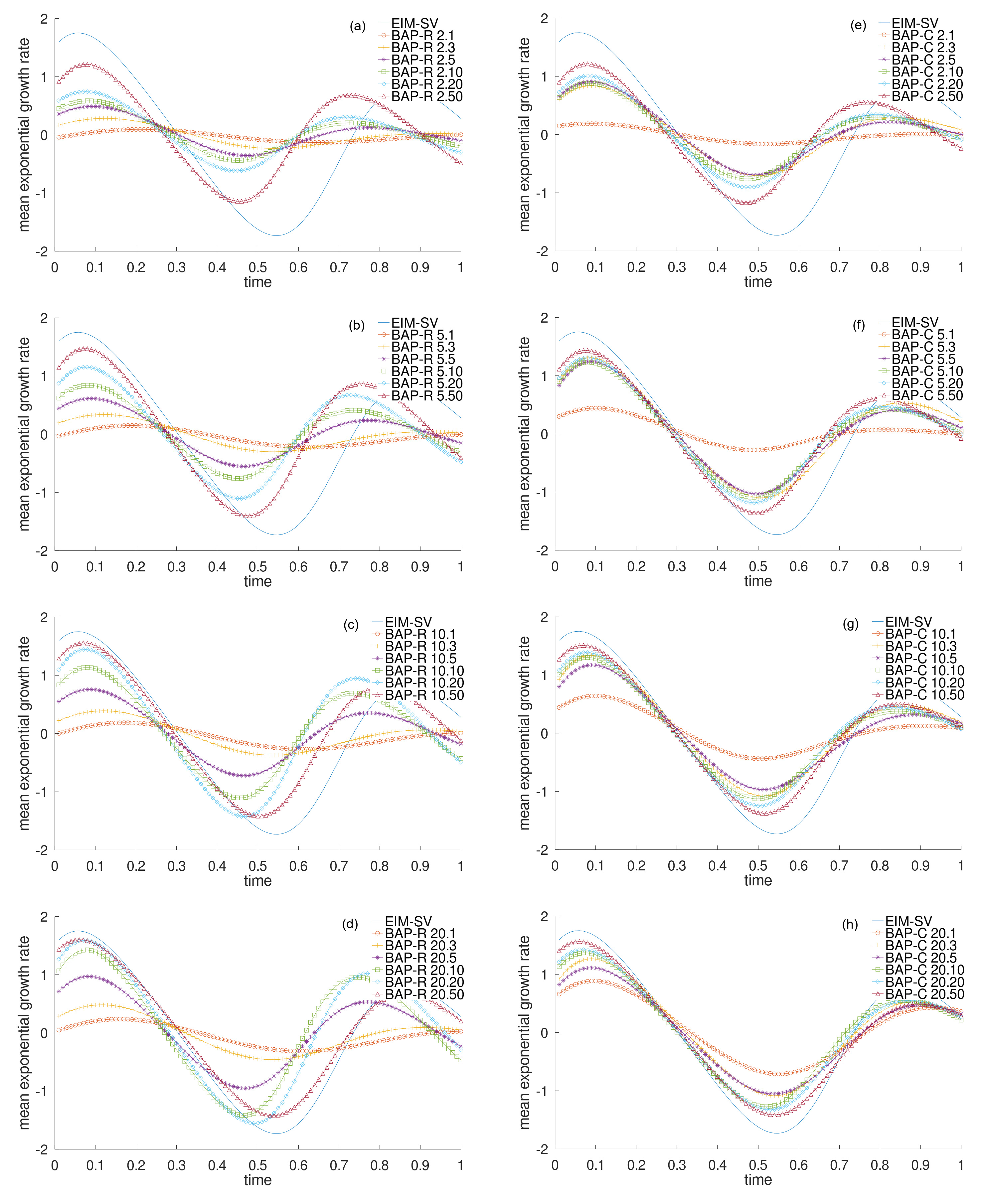}
  \caption[Bildunterschrift]{MEGR for SWM - EIM-SV and \underline{r}andom vector started BAP (BAP-R) (left hand side) or \underline{c}hord vector started BAP (BAP-C) (right hand side). BAP-R/C i.j denotes an approximation of EIM-SV obtained by the usage of i initial vectors and j iteration loops for BAP. Plot (a) BAP-R $i=2$, (b) BAP-R $i=5$, (c) BAP-R $i=10$, (d) BAP-R $i=20$, (e) BAP-C $i=2$, (f) BAP-C $i=5$, (g) BAP-C $i=10$ and (h) BAP-C $i=20$. }\label{fig:SWM_GR}
\end{figure}

\clearpage

In contrast to L96 the version of the SWM used here is large enough to get significant benefits from BAP with respect to the computational costs. Table \ref{tab:SWM_time} shows the low costs of BAP relative to EIM-SV. For the BAP with block size of four and ten iteration loops the computational costs are at only 0.024 relative to the full computation of EIM-SV, while the corresponding REGI reaches 0.75.

We like to point out, that the full computation of EIM-SV is very expensive and can not be compared to the classical SV approximation with a linearised model and the Lanczos algorithm but BAP are several orders of magnitude cheaper. 
We are convinced that BAP is a method of large potential to generate fast growing perturbations in high dimensional dynamical systems, but it is not exactly clear how BAP compares to classical SV.

\section{Summary and Conclusions}

A new parallelisable method for generating fast growing perturbations in complex dynamical systems, especially for estimating uncertainty  using ensemble prediction systems (EPS) in numerical weather forecasting,
has been introduced in this paper. 

It is based on the matrix-free approximation of the singular vectors of an evolved increment matrix (EIM) with a block Arnoldi method. This concept is different from the classical approach of computing singular vector perturbations (SV) for ensemble forecasting, implemented e.g. at ECMWF (\cite{bib:Leutb_Palmer08}), which estimates the leading singular vectors a tangent linear propagator. In contrast to classical SV, this block Arnoldi perturbation (BAP) method requires short forecasts with the full non-linear model only, covers also non-linear effects partly and saves the costs of setting up and maintaining linear and adjoint model versions.

We observed promising results for the Lorenz96 (L96) system with 50 and for a shallow water model (SWM) with 1587 degrees of freedom. BAP were tested with two different kinds of initial vectors, namely random vectors and differences of past states from the reference trajectory (chord vectors). In the more complex SWM the convergence of the block Arnoldi approximation improves especially for smaller subspaces when the algorithm is started by chord vectors instead of random initial vectors.  The results show, that the choice of the initial vectors is a relevant part of the method.

We classify an approximation of strongly growing perturbations, as an important subgoal for the generation of reliable ensemble prediction systems. Further research is needed for the questions of how to exactly create an ensemble prediction system using BAP and how such a system works for NWP especially in direct comparison to classical SV.

We are convinced that the BAP approach has the potential to become a relative cheap and relative easy applicable method for identifying the local expanding subspaces in complex dynamical systems.

\section*{Acknowledgement}

We like to thank Stephan Dahlke and Alexander Sieber at the department Numerics and Optimization of the Phillips Universit\"at Marburg for their useful suggestions and help in preparing this article. 
We also thank the BMVI - Expertennetzwerk (German Federal Ministry of Transport and Digital Infrastructure - Network of Experts) for the financial support.

%+++++++++++++++++++++++++++++++++++++++++++
\clearpage

%++++++++++++++++++++++++++++++++++++++++++++++

%\renewcommand{\bibname}{Literaturverzeichnis}
%\bibliographystyle{amsalpha} % Nummerierung mit Buchstaben (gut bis etwa 20 Angaben)
%\bibliographystyle{amsplain} % Nummerierung mit Zahlen (besser ab etwa 20 Angaben)
%\bibliographystyle{abbrv} % Nummerierung mit Zahlen (besser ab etwa 20 Angaben)
\bibliographystyle{apalike}
\bibliography{literatur_pert_b}

\clearpage
%\begin{biography}[example-image-1x1]{A.~One}
%Please check with the journal's author guidelines whether author biographies are required. They are usually only included for review-type articles, and typically require photos and brief biographies (up to 75 words) for each author.
%\bigskip
%\bigskip
%\end{biography}

%\graphicalabstract{example-image-1x1}{Please check the journal's author guildines for whether a graphical abstract, key points, new findings, or other items are required for display in the Table of Contents.}

%

% Comment: Table with REGI-data side by side (version L96-A)

%

\end{document}